\documentclass[12pt]{amsart}

\usepackage{amssymb}
\usepackage{amsthm}
\usepackage{amsmath}
\usepackage{amscd}

\usepackage{verbatim}
\usepackage[all]{xy}
\usepackage{diagbox}
\usepackage{dsfont}
\usepackage{fancyhdr}
\usepackage{bbding}
\usepackage{tikz}

\numberwithin{equation}{section}

\theoremstyle{plain}
\newtheorem{theorem}{Theorem}[section]
\newtheorem{corollary}[theorem]{Corollary}
\newtheorem{lemma}[theorem]{Lemma}
\newtheorem{proposition}[theorem]{Proposition}

\newtheorem{conjecture}[theorem]{Conjecture}
\theoremstyle{definition}
\newtheorem{definition}[theorem]{Definition}
\newtheorem{remark}[theorem]{Remark}

\theoremstyle{remark}

\newcommand{\OO}{\mathcal O}
\newcommand{\A}{\mathbb{A}}
\newcommand{\R}{\mathbb{R}}

\newcommand{\Q}{\mathbb{Q}}
\newcommand{\Z}{\mathbb{Z}}

\newcommand{\C}{\mathbb{C}}

\renewcommand{\H}{\mathbb{H}}

\newcommand{\D}{\mathbb{D}}

\newcommand{\kzxz}[4]{\left(\begin{smallmatrix} #1 & #2 \\ #3 & #4\end{smallmatrix}\right) }

\newcommand{\vol}{\operatorname{vol}}
\newcommand{\tr}{\operatorname{tr}}

\newcommand{\Cl}{\operatorname{Cl}}

\newcommand{\Mp}{\operatorname{Mp}}

\newcommand{\End}{\operatorname{End}}

\newcommand{\cha}{\operatorname{char}}

\newcommand{\SL}{\operatorname{SL}}

\newcommand{\CH}{\operatorname{CH}}

\begin{document}
\setcounter{section}{-1}

\title{On the central  derivatives of L-functions and modularity of  Heenger cycles}
\author[Tuoping Du, Zhifeng Peng]{Tuoping Du,\ Zhifeng Peng}

\footnotetext[1]{Tuoping Du, \Envelope dtp1982@163.com, \Envelope tpdu@ncepu.edu.cn, ~School of Mathematics and Physics, North China  Electric Power University, Beijing, 102206, P.R. China.}
\footnotetext[2]{Zhifeng Peng, \Envelope zfpeng@suda.edu.cn,  ~School of Mathematics, Soochow university, Suzhou, 215006, P.R. China.}


\begin{abstract}

This paper establishes an arithmetic intersection formula for central $L$-derivatives in higher weight. We prove that for a general cusp form (extending the previous result for newforms), the derivative  is represented by the global height pairing between higher Heegner cycles. This result provides a framework for the Gross–Zagier–Zhang formula and its generalizations. Furthermore, we investigate the modularity of the generating series of Heegner cycles, proving a weak version of the conjecture and reducing the full modularity to a vanishing conjecture, for which we provide supporting evidence.

\end{abstract}

\dedicatory{}

\date{\today}

\maketitle


\section{Introduction}

The study of central derivatives of $L$-functions lies at the heart of modern number theory. A cornerstone of this field is the Gross–Zagier formula \cite{GZ}, which connects the first derivative of an elliptic curve's $L$-function at the central point to the height of a Heegner point. This result has since become a fundamental tool in the study of the Birch and Swinnerton-Dyer (BSD) conjecture, providing a powerful geometric approach to understanding ranks of elliptic curves. Zhang's later extension \cite{Zhang} of this framework to higher-weight newforms further broadened its scope, expressing central derivatives of $L$-functions in terms of heights of Heegner cycles.

Building upon this geometric foundation, Bruinier and Yang \cite{BY} pioneered a new direction by studying the central derivatives of $L$-functions associated with more general cusp forms. Later, a new arithmetic representation conjecture was proposed in both \cite{BEY} and \cite{Du}. In this paper, we focus on the case of the trivial character, where the underlying structure allows for a more explicit analysis.

Within this setting, we make three main contributions. First, we formulate a precise modularity conjecture for an arithmetic generating function. Second, we prove a weak version of this conjecture, showing that the full modularity property follows from the vanishing of certain Heegner cycles. Third, we establish an explicit formula that provides direct evidence for a related vanishing conjecture. By exploring these three interconnected questions, we aim to advance the understanding of central $L$-derivatives and their role in the broader arithmetic landscape.
%
%
%

We define the quadratic space
\begin{equation}
V=\{w=\kzxz
    {w_1} { w_2 }
      {w_3}   {-w_1}  \in M_{2}(\Q) : \tr(w) =0 \},
\end{equation}
equipped with the quadratic form $Q(w) = N \cdot \det w$. We then define the lattice
 \begin{equation}
L=\big\{w =\kzxz {b}{\frac{a}{N}}{c}{-b}
   \in M_{2}(\Z) :\,   a, b, c \in \Z \big\}\subset V.
\end{equation}
Let $L^\sharp$ denote the dual lattice of $L$, and let $\Gamma' = \Mp_2(\Z)$ be the full inverse image of $\SL_2(\Z)$ under the two-fold metaplectic cover of $\SL_2(\R)$.

The Weil representation $\rho_L$ acting on the finite-dimensional group algebra $\mathbb{C}[L^\sharp / L]$, which has a standard basis $\{ e_\mu \mid \mu \in L^\sharp / L \}$.

Let $H_{k, \rho_L}$ be the space of harmonic weak Maass forms of weight $k$ with respect to $\rho_L$. The differential operator $$\xi_{k}: H_{k, \rho_L } \rightarrow S_{2-k, \overline{\rho}_L}$$
maps such forms to cusp forms of weight $2 - k$ and contragredient representation $\overline{\rho}_L$.

For each $\mu \in L^\sharp / L$ and each positive rational number $m \in Q(\mu) + \mathbb{Z}$, we define the Heegner divisor $Z(m, \mu)$ on the modular curve $ X_0(N)$, following Bruinier and Ono \cite{BO}. This generalizes the Heegner divisors originally studied by Gross and Zagier \cite{GZ}.

%
%

The Gross–Zagier formula \cite{GZ} was extended to higher weight by Zhang \cite{Zhang}, who showed that for any eigenform $G$ of weight $2\kappa$, the central derivative $L'(G, \chi, \kappa)$ can be expressed in terms of heights of Heegner cycles.

Let $\mathcal{Y} = \mathcal{Y}\kappa(N)$ denote the Kuga–Sato variety. For any CM point $x$ on $X_0(N)$, let $S\kappa(x)$ denote the Heegner cycle over $x$ in $\mathcal{Y}$. In \cite{Zhang}, Zhang studied the global height pairing $\langle \cdot, \cdot \rangle$ for higher Heegner cycles, which generalizes the Gillet–Soulé height introduced in \cite{GS}. This pairing decomposes as a sum of local height pairings.

The $\kappa$-th Heegner cycle is defined as in \cite{BEY} and \cite{Zhang} by
\begin{equation}\label{defheer}
Z_\kappa(m, \mu) = m^{\frac{\kappa - 1}{2}} \sum_{x \in Z(m, \mu)} S_\kappa(x) \in \CH^\kappa(\mathcal{Y}).
\end{equation}
For any negative definite two-dimensional subspace $U \subseteq V$, the Heegner cycle $Z_\kappa(U)$ is defined analogously.

Given $f \in H_{3/2-\kappa, \bar{\rho}L}$, we define the associated Heegner cycle by
\begin{equation}
Z{\kappa}(f) = \sum_{m > 0} c^+(-m, \mu) Z_{\kappa}(m, \mu),
\end{equation}
where $c^+(m, \mu)$ are the Fourier coefficients of the holomorphic part $f^+$ of $f$.
Let $\mathrm{Heeg}_{\kappa}(X)$ denote the space of Heegner cycles on $X = X_0(N)$. Define $\mathrm{Pin}$ as the subgroup of principal cycles given by rational functions, and $\mathrm{Pin^{alg}}$ as the subgroup of algebraic principal cycles, i.e., those defined over divisors coming from canonical meromorphic modular forms \cite{ABS} with algebraic coefficients \cite{Ze}. The corresponding quotients are denoted $\mathrm{Heeg}_{\kappa}(\Cl(X))$ and $\mathrm{Heeg}_{\kappa}^{\mathrm{alg}}(\Cl(X))$, respectively.

Let $G \in S_{2\kappa}(\Gamma_0(N))$ be a normalized newform. Under the Shimura correspondence $\mathrm{Sh}{m_0,\mu_0}$, there exists a newform $g \in S_{\frac{1}{2}+\kappa, \rho_L}$ corresponding to $G$, as well as a function $f \in H_{3/2-\kappa, \bar{\rho}_L}$ such that
$$\xi_{3/2-\kappa}(f)=\frac{1}{\langle g,~ g \rangle_{Pet}}g.$$

We now define the generating function
\begin{equation}\phi_\kappa=\sum_{m>0, \mu}Z_\kappa(m, \mu)q^me_\mu \in \mathrm{Heeg_{{\kappa}}(X)}\otimes \C[[q]].
  \end{equation}

The following conjecture is the main objective of this paper.
\begin{conjecture}[Modularity]\label{confirst}
Let the notation be as above. 
\begin{enumerate}
  \item The  function $\phi_\kappa\in \mathrm{\CH^{{\kappa}}(\mathcal{Y})}\otimes S_{1/2+\kappa, \rho_L}$ is a cusp form, where $\CH^{{\kappa}}(\mathcal{Y})$ is the codimension $\kappa$ Chow group.
  \item The $G$-component is given by
  \begin{equation}
\phi_\kappa^G=g(\tau)\otimes Z_{ \kappa}(f).
\end{equation}
\end{enumerate}
\end{conjecture}

\begin{remark}
\
\begin{itemize}
    \item The same modularity conjecture applies to the arithmetic generating function
    \[
    \hat{\phi}_\kappa = \sum_{m > 0,\, \mu} \widehat{\mathcal{Z}}_\kappa(m, \mu)  q^m e_\mu \in \widehat{\mathrm{CH}}^{\kappa}(\mathcal{Y}) \otimes \mathbb{C}[[q]].
    \]

    \item For $\kappa = 1$, the modularity of Heegner divisors has been established by Gross–Kohnen–Zagier \cite{GKZ}, Borcherds \cite{BoDuke}, and Bruinier–Ono \cite{BO}.
\end{itemize}
\end{remark}

In this work, we establish the following modularity result.
\begin{theorem}\label{thefir}
If the Chow group $\CH^{{\kappa}}(\mathcal{Y})$ is replaced by $\mathrm{Heeg_{{\kappa}}^{alg}(\Cl (\mathrm{X}))}$, then Conjecture \ref{confirst} holds.
\end{theorem}

The proof shows that Conjecture \ref{confirst} reduces to establishing the vanishing of $Z_{\kappa}(f)$ in the Chow group, as formulated in the following statement.

\begin{conjecture}\label{conjsec}
For any weakly holomorphic modular form $f$, the Heegner cycle $Z_{\kappa}(f)$ vanishes in the Chow group.
\end{conjecture}

We outline the logical dependence of the main conjecture as follows.
\begin{proposition}
$$\text{Conjecture}~ \ref{conjsec}+\text{Theorem} ~\ref{thefir}\Rightarrow ~\text{Conjecture}~ \ref{confirst}.$$
\end{proposition}

Thus, the validity of Conjecture \ref{confirst} depends on the vanishing of $Z_{\kappa}(f)$. Although a full proof of Conjecture \ref{conjsec} is not completed in this work, we establish that $\langle Z_{\kappa}(f), Z_{\kappa}(U) \rangle = 0$ for almost all Heegner cycles $Z_{\kappa}(U)$ in Corollary \ref{corzero}.

Identifying $Z_{\kappa}(U) = Z_{\kappa}(m_0, \mu_0)$ with fundamental discriminant $D_0 = -4Nm_0$, we obtain the following result.
\begin{theorem}\label{maintheo}
 For any $f \in H_{3/2-\kappa, \bar{\rho}_L}$, if $(D_0, 2ND(f))=1$,   then the global height is given by
\begin{equation}
\langle  Z_{ \kappa}(f),  Z_{\kappa}(U)\rangle
=\frac{2\sqrt{N}\Gamma(\kappa-\frac{1}{2})}{(4\pi)^{\kappa-1}\pi^{\frac{3}{2}}}L'(Sh_{m_0, \mu_0}(\xi_{3/2-\kappa}f),  \kappa),
\end{equation}
where $Sh_{m_0, \mu_0}(\xi_{3/2-\kappa}f)$ is a cusp form of weight $2\kappa$, and $D(f)$ denotes the discriminant of $Z_\kappa(f)$.
\end{theorem}
\begin{remark}
\
\begin{enumerate}
\item The case where $\kappa$ is odd was proved in \cite[Section 6]{BEY}.
\item The twisted case has been studied by the author in \cite{Du}.
\end{enumerate}
\end{remark}

\begin{corollary}\label{corzero}
 For any $f \in M^!_{3/2-\kappa, \bar\rho_L}$, we have
\begin{equation}
\langle  Z_{\kappa}(f),  Z_{ \kappa}(U)\rangle=0.
\end{equation}
\end{corollary}

This corollary provides a weak version of Conjecture \ref{conjsec}, i.e., $ Z_{\kappa}(f)$ vanishes for any $f \in M^!_{3/2-\kappa, \bar\rho_L}$.

\begin{proposition}
\begin{equation}
\text{Theorem \ref{maintheo} }\overset{\text{Modularity}}{\Longrightarrow} \text{Gross-Zagier-Zhang formua}.
\end{equation}
\end{proposition}

\begin{theorem}[Gross-Zagier-Zhang formula in \cite{Zhang}]
\begin{equation}
\langle Z_{ \kappa}^G(m, \mu), Z_{\kappa}^G(m, \mu)\rangle=
\frac{(2\kappa-2)!\sqrt{N} |m_0|^{\kappa-\frac{1}{2}}}{2^{4\kappa-3}\pi^{\kappa+1}\parallel G\parallel^2}
L'(G,  \kappa).
\end{equation}
\end{theorem}
\begin{remark}
 This work treats only the case of trivial character. The above formula generalizes the classical Gross--Zagier formula \cite{GZ}, which corresponds to the case $\kappa = 1$.
\end{remark}

\subsection{Plan of the Proof}
\

In Section \ref{Sec1}, we review foundational results on modular forms and Heegner divisors.

In Section \ref{Sec2}, we introduce the regularized theta lift, with particular emphasis on the lift associated to non-holomorphic Hejhal-Poincaré series. This construction yields higher Green functions for higher Heegner cycles.

In Section \ref{Sec3}, we recall the moduli stacks of Heegner divisors and the height pairing of Heegner cycles, as developed by Zhang \cite{Zhang}. This global pairing decomposes as a sum of local intersection numbers.

In Section \ref{Sec4}, we compute the local intersection numbers $\langle Z_{\kappa}(f), Z_{\kappa}(U) \rangle_p$ at each prime $p$, including both archimedean and non-archimedean contributions.

In Section \ref{Sec5}, we combine the local results to complete the proof of Theorem \ref{maintheo}, thereby obtaining an explicit formula for the global height pairing $\langle Z_{\kappa}(f), Z_{\kappa}(U) \rangle$.

In Section \ref{Sec6}, we investigate the modularity of Heegner cycles and prove Theorem \ref{thefir}. Furthermore, we reduce Conjecture \ref{confirst} to Conjecture \ref{conjsec}, establishing that the modularity property hinges on the vanishing of $Z_{\kappa}(f)$ for all weakly holomorphic modular forms $f$.

%
%
%
%
%
%
%
%

\section{Preliminaries}\label{Sec1}

Let $\widetilde{\SL}_2(\R)$ be the metaplectic double cover of $\SL_2(\R)$, consisting of pairs $(g, \phi(g, \tau))$, where $g = \begin{pmatrix}a & b \ c & d\end{pmatrix} \in \SL_2(\R)$ and $\phi(g, \tau)$ is a holomorphic function on the upper half-plane $\H$ satisfying $\phi(g, \tau)^2 = j(g, \tau) = c\tau + d$. Let $\Gamma' = \Mp_2(\Z)$ be the preimage of $\Gamma = \SL_2(\Z)$ in $\widetilde{\SL}_2(\R)$. Then $\Gamma'$ is generated by
$$
S=\left( \kzxz {0} {-1} {1} {0}, \sqrt \tau \right)  \quad  T= \left( \kzxz {1} {1} {0} {1} , 1 \right).
$$

Let $(V, Q)$ be a quadratic space of signature $(p, q)$ and $L \subset V$ an even lattice. Denote by $L^\sharp$ the dual lattice of $L$. The quadratic form on $L$ induces a $\Q/\Z$-valued quadratic form on the discriminant group $L^\sharp / L$. Let ${e_\mu \mid \mu \in L^\sharp / L}$ be the standard basis of the group algebra $S_L = \C[L^\sharp / L]$. The Weil representation $\rho_L$ of $\Gamma'$ on $\C[L^\sharp / L]$ (see \cite{Bo}) is defined by
\begin{eqnarray} \label{eq:WeilRepresentation}
&\rho_{L} (T)e_{\mu}&=e(Q(\mu))e_{\mu} ,\\
&\rho_{L}(S) e_{\mu}&= \frac{e(-\frac{p-q}8)}{\sqrt{\vert L^{\sharp} / L \vert}}\sum \limits_{\mu^{\prime} \in L^{\sharp} / L }e(-(\mu,
\mu^{\prime}))e_{\mu^{\prime}}\nonumber .
\end{eqnarray}

\subsection{Modular Forms}

A twice continuously differentiable function $f: \mathbb{H} \rightarrow \mathbb{C}[L^{\sharp}/L]$ is called a weak Maass form of weight $k \in \frac{1}{2}\mathbb{Z}$ and representation $\rho_L$ if it satisfies:
\begin{enumerate}
    \item $f \mid_{k, \rho_L} \gamma = f$ for all $\gamma \in \Gamma'$;
    \item There exists $\lambda \in \mathbb{C}$ such that $\Delta_k f = \lambda f$;
    \item There exists $C > 0$ such that $f(\tau) = O(e^{Cv})$ as $v \rightarrow \infty$, uniformly in $u$.
\end{enumerate}

The slash operator is defined by:
\[
f \mid_{k, \rho_L} \gamma(\tau) = \phi(\tau)^{-2k} \rho_L^{-1}(\gamma') f(\gamma \tau).
\]

When $\lambda = 0$, $f$ is called a harmonic weak Maass form. Denote by $H_{k, \rho_L}$ the space of such forms. The principal part of $f$ is defined as:
\[
P_f(\tau) = \sum_{\mu} \sum_{n \leq 0} c(n, \mu) q^n e_\mu.
\]

Every $f \in H_{k, \rho_L}$ admits a unique decomposition $f = f^+ + f^-$ with Fourier expansions:
\begin{align}
f^+ & = \sum_{\mu} \sum_{n \gg -\infty} c^+(n, \mu) e(n\tau) e_\mu, \\
f^- & = \sum_{\mu} \sum_{n < 0} c^-(n, \mu) \Gamma(1 - k, 2\pi |n|v) e(n\tau) e_\mu,
\end{align}
where $\Gamma(a, x) = \int_x^\infty e^{-t} t^{a-1} dt$ is the incomplete Gamma function.

For any field $F$, let $H_{k, \rho_L}(F)$ denote the subspace of forms whose principal part has coefficients in $F$.

The $\xi$-operator $\xi_k: H_{k, \overline{\rho}_L} \rightarrow S_{2-k, \rho_L}$ is defined by:
\[
\xi_k(f) = 2i v^k \frac{\overline{\partial f}}{\partial \bar{\tau}},
\]
giving the exact sequence:
\[
0 \rightarrow M_{k, \overline{\rho}_L}^! \rightarrow H_{k, \overline{\rho}_L} \xrightarrow{\xi_k} S_{2-k, \rho_L} \rightarrow 0,
\]
where $M_{k, \overline{\rho}_L}^!$ is the space of weakly holomorphic modular forms.

The Maass lowering and raising operators are defined as:
\[
L_k = -2i v^2 \frac{\partial}{\partial \bar{\tau}}, \quad R_k = 2i \frac{\partial}{\partial \tau} + k v^{-1}.
\]

The standard scalar product on $\mathbb{C}[L^{\sharp}/L]$ is:
\[
\left\langle \sum_{\mu} f_\mu e_\mu, \sum_{\mu} g_\mu e_\mu \right\rangle = \sum_\mu f_\mu g_\mu.
\]

For $f, g \in M_{k, \rho_L}$, the Petersson scalar product is:
\[
\langle f, g \rangle_{\mathrm{Pet}} = \int_{\mathcal{F}} \langle f(\tau), \overline{g(\tau)} \rangle v^k \, \mu(\tau),
\]
where $\mu(\tau) = \frac{du \, dv}{v^2}$ is the hyperbolic measure.

Let $\Gamma'' \subseteq \Gamma'$ be a congruence subgroup and $j \in \mathbb{Z}_{\geq 0}$. For $f \in M_k(\Gamma'')$ and $g \in M_l(\Gamma'')$, the $j$-th Rankin-Cohen bracket is:
\[
[f, g]_j = \sum_{s=0}^j (-1)^s \binom{k + j - 1}{s} \binom{l + j - 1}{j - s} f^{(j-s)} g^{(s)} \in M_{k + l + 2j}(\Gamma''),
\]
where $g^{(s)} = \frac{1}{(2\pi i)^s} \frac{\partial^s g}{\partial \tau^s}$.

\subsection{Eisenstein Series}

Let $V$ be a quadratic space of dimension $m$. Define the character
\[
\chi_V(x) = (x, (-1)^{\frac{m(m-1)}{2}} \det(V))_{\mathbb{A}},
\]
where $\det(V)$ denotes the Gram determinant of $V$.

For any standard section $\Phi(g, s)$ in the induced representation $I(s, \chi_V)$, define the Eisenstein series by
\[
E(g, s, \Phi) = \sum_{\gamma \in P(\mathbb{Q}) \backslash \widetilde{\SL}_2(\mathbb{Q})} \Phi(\gamma g, s),
\]
where $P(\mathbb{Q})$ is the parabolic subgroup.

There exists an equivalent map
\[
S(V) \rightarrow I(s_0, \chi_V), \quad s_0 = \frac{m}{2} - 1,
\]
such that
\[
\Phi(g, s_0) = \omega(g) \varphi(0),
\]
where $\omega$ denotes the Weil representation.

Let $\Phi^l_\infty$ be the unique archimedean standard section satisfying
\[
\Phi^l_\infty\left( \begin{pmatrix} \cos\theta & \sin\theta \\ -\sin\theta & \cos\theta \end{pmatrix}, s \right) = e^{il\theta}, \quad \theta \in [0, 2\pi].
\]

For $\mu \in L^\sharp / L$, define $\varphi_\mu = \text{char}(L_\mu) \in S(V(\mathbb{A}_f))$. The weight $l$ Eisenstein series is defined by
\[
E_L(\tau, s; l) = v^{-\frac{l}{2}} \sum_\mu E(g_\tau, s, \Phi^l_\infty \otimes \lambda(\varphi_\mu)) \varphi_\mu.
\]

Identifying $\varphi_\mu$ with $e_\mu$, it follows from \cite[Section 2]{BY} that
\[
E_L(\tau, s; l) = \sum_{\gamma' \in \Gamma_\infty' \backslash \Gamma'} \left( v^{\frac{s+1-l}{2}} e_0 \right) \mid_{l, \rho_L} \gamma'.
\]

When $V$ is a quadratic space of signature $(0,2)$, the central value $E_L(\tau, s; 1)$ vanishes at $s_0 = 0$, while $E_L(\tau, s; -1)$ is holomorphic at $s_0$. These Eisenstein series satisfy the relation
\[
L_1 E_L(\tau, s; 1) = \frac{s}{2} E_L(\tau, s; -1),
\]
where $L_l = -2i v^2 \frac{\partial}{\partial \bar{\tau}}$ is the Maass lowering operator of weight $l$.

The derivative
\[
E'_L(\tau, s_0; 1) = \sum_{\mu \in L^\sharp / L} \sum_n b_\mu(n, v) q^n e_\mu
\]
is a harmonic weak Maass form of weight $1$. Denote its holomorphic part by
\begin{equation}\label{equFE}
\mathcal{E}_L(\tau) = \sum_\mu \sum_{n \geq 0} \kappa(n, \mu) q^n e_\mu.
\end{equation}

The coefficients are given explicitly (see \cite{KuIntegral, Scho}) by
\[
\kappa(n, \mu) = 
\begin{cases}
\lim_{v \to \infty} b_0(0, v) - \log v, & \text{if } n = 0 \text{ and } \mu = 0, \\
\lim_{v \to \infty} b_\mu(n, v), & \text{otherwise}.
\end{cases}
\]

Note that $\kappa(n, \mu) = 0$ for $n < 0$, and $\kappa(0, \mu) = 0$ for $\mu \neq 0$. These coefficients have been computed via local methods in \cite{BY}.

\subsection{Heegner Divisors}

In this work, we consider the quadratic space
\[
V = \left\{ w = \begin{pmatrix} x_1 & x_2 \\ x_3 & -x_1 \end{pmatrix} \in M_2(\mathbb{Q}) \right\},
\]
equipped with the quadratic form $Q(w) = N \cdot \det(w)$. Define the lattice
\[
L = \left\{ w = \begin{pmatrix} b & -\frac{a}{N} \\ c & -b \end{pmatrix} \in M_2(\mathbb{Z}) \,\middle|\, a, b, c \in \mathbb{Z} \right\},
\]
and its dual lattice
\[
L^\sharp = \left\{ w = \begin{pmatrix} \frac{b}{2N} & -\frac{a}{N} \\ c & -\frac{b}{2N} \end{pmatrix} \in M_2(\mathbb{Z}) \,\middle|\, a, b, c \in \mathbb{Z} \right\}.
\]

We identify
\[
\mathbb{Z}/2N\mathbb{Z} \cong L^\sharp / L, \quad r \mapsto \mu_r = \begin{pmatrix} \frac{r}{2N} & 0 \\ 0 & -\frac{r}{2N} \end{pmatrix}.
\]

Let $\D$ be the Hermitian domain of oriented negative 2-dimensional subspaces of $V(\mathbb{R})$. We identify $\D$ with $\mathbb{H} \cup \overline{\mathbb{H}}$ via
\[
z = x + iy \mapsto \mathbb{R} \cdot \mathfrak{R} \begin{pmatrix} z & -z^2 \\ 1 & -z \end{pmatrix} + \mathbb{R} \cdot \mathfrak{I} \begin{pmatrix} z & -z^2 \\ 1 & -z \end{pmatrix}.
\]

For $w = \begin{pmatrix} \frac{b}{2N} & -\frac{a}{N} \\ c & -\frac{b}{2N} \end{pmatrix} \in L^\sharp$, the associated CM point is
\[
z(w) = \frac{b}{2Nc} + \frac{\sqrt{b^2 - 4Nac}}{2N|c|} \in \mathbb{H}.
\]

For $\mu \in L^\sharp / L$ and $m \in Q(\mu) + \mathbb{Z}$ with $m > 0$, let
\[
L_\mu = L + \mu, \quad L_\mu[m] = \{ w \in L_\mu \mid Q(w) = m \}.
\]

The Heegner divisor on $X_0(N)$ is defined by
\begin{equation}\label{TwistedHeegner}
Z(m, \mu) := \sum_{w \in \Gamma_0(N) \setminus L_\mu[m]} z(w),
\end{equation}
which is defined over $\mathbb{Q}$.

Now let $U \subset V$ be a quadratic subspace of signature $(0,2)$, identified with an imaginary quadratic field $k = \mathbb{Q}(\sqrt{D_0})$. The two points $\{z_U^\pm\}$ correspond to the two orientations of $U(\mathbb{R})$.

The CM divisor is defined as
\[
Z(U) = k^\times \backslash \{z_U^\pm\} \times \mathbb{A}_{k,f}^\times / \widehat{\mathcal{O}}_k^\times \cong \{z_U^\pm\} \times \mathrm{Cl}_k,
\]
where $\mathrm{Cl}_k$ denotes the ideal class group of $k$.

When $D_0 = -4Nm_0$ is a fundamental discriminant, we have
\[
Z(U) \simeq Z(m_0, \mu_0),
\]
so $Z(U)$ is also a Heegner divisor.

\section{Regularized theta lift}\label{Sec2}
\subsection{Theta Functions}

For each $z \in \D$, we have the orthogonal decomposition
\[
V(\mathbb{R}) = z \oplus z^\perp.
\]
For any $w \in V(\mathbb{R})$, write $w = w_z + w_{z^\perp}$ accordingly.

Define the associated majorant
\[
(w, w)_z = (w_{z^\perp}, w_{z^\perp}) - (w_z, w_z),
\]
which is a positive definite quadratic form on $V(\mathbb{R})$. Consider the Gaussian
\[
\varphi_\infty(w, z) = e^{-\pi (w, w)_z} \in S(V(\mathbb{R})).
\]

For $\tau = u + iv \in \mathbb{H}$, let $g_\tau = \begin{pmatrix} 1 & u \\ 0 & 1 \end{pmatrix} \begin{pmatrix} v^{1/2} & 0 \\ 0 & v^{-1/2} \end{pmatrix}$. Under the Weil representation, define
\[
\varphi(w, \tau, z) = v^{1/4} \omega(g_\tau) \varphi_\infty(w, z) = v \cdot e\left( Q(w_{z^\perp}) \tau + Q(w_z) \bar{\tau} \right).
\]

The Siegel theta function is defined by
\[
\Theta(\tau, z) = \sum_{\mu \in L^\sharp / L} \sum_{w \in L_\mu} \varphi(w, \tau, z) e_\mu = v \sum_{\mu \in L^\sharp / L} \sum_{w \in L_\mu} e\left( Q(w_{z^\perp}) \tau + Q(w_z) \bar{\tau} \right).
\]
This is a modular form of weight $-\frac{1}{2}$, satisfying
\[
\Theta(\gamma \tau, z) = \phi(\tau)^{-1} \rho_L(\gamma') \Theta(\tau, z).
\]

More generally, for a quadratic space $V$ of signature $(n, 2)$, one can define a Siegel theta function of weight $\frac{n}{2} - 1$.

For $z = x + yi \in \mathbb{H}$, define the normalized vector
\[
\lambda(z) = \frac{1}{2\sqrt{N} y} \begin{pmatrix} -x & x^2 + y^2 \\ -1 & x \end{pmatrix}.
\]
For $w = \begin{pmatrix} \frac{b}{2N} & -\frac{a}{N} \\ c & -\frac{b}{2N} \end{pmatrix} \in L^\sharp$, define
\[
p_z(w) = -(w, \lambda(z)) = -\frac{1}{2\sqrt{N} y} \left( Nc|z|^2 - bx + a \right).
\]

The Millson theta function \cite{KM2} is defined as
\[
\Theta^{\mathcal{M}}(\tau, z) = v \sum_{\mu \in L^\sharp / L} \sum_{w \in L_\mu} p_z(w) e\left( Q(w_{z^\perp}) \tau + Q(w_z) \bar{\tau} \right) e_\mu,
\]
which is a modular form of weight $1/2$ transforming with the representation $\rho_L$.

Finally, the representation $\rho_L$ can be identified with the restriction to $\Gamma'$ of the complex conjugate of the Weil representation $\omega$ on $S(V(\mathbb{A}_f))$.

\subsection{Regularized Theta Lift}

Let $V$ be a quadratic space of signature $(n, 2)$, and let $f \in H_{1 - n/2, \bar\rho_L}$. The regularized theta integral is defined by
\[
\Phi(z, f) = \int_{\mathcal{F}}^{\mathrm{reg}} \langle f(\tau), \Theta(\tau, z) \rangle \, d\mu(\tau),
\]
interpreted as the constant term in the Laurent expansion.

More precisely, for $f \in H_{1/2, \bar\rho_L}$, define
\[
\Phi(z, f) = \mathrm{CT}_{s = 0} \lim_{T \to \infty} \int_{\mathcal{F}_T} \langle f(\tau), \Theta(\tau, z) \rangle \frac{1}{v^s} \, d\mu(\tau),
\]
where $\mathcal{F}_T = \{ \tau \in \mathbb{H} \mid |u| \leq \tfrac{1}{2},\; v < T,\; |\tau| > 1 \}$ is the truncated fundamental domain.

For $f \in H_{-1/2, \bar\rho_L}$, the regularized Millson theta integral \cite{BEY} is defined by
\[
\Phi^{\mathcal{M}}(z, f) = \int_{\mathcal{F}}^{\mathrm{reg}} \langle f(\tau), \Theta^{\mathcal{M}}(\tau, z) \rangle \, d\mu(\tau).
\]

For $f \in H_{1/2, \bar{\rho}_{L}}$, define the Heegner divisor
\[
Z(f) = \sum_{m < 0,\, \mu} c^+(m, \mu) Z(m, \mu),
\]
where $c^+(m, \mu)$ are the Fourier coefficients of the holomorphic part $f^+$.

\begin{theorem}[\cite{BO}, Proposition 5.2]
Let $f \in H_{1/2, \bar{\rho}_{L}}$. Then $\Phi(z, f)$ is smooth on $Y_0(N) \setminus Z(f)$ and has logarithmic singularities along the divisor $-2Z(f)$. Moreover, if $\Delta_z$ denotes the invariant Laplacian on $\mathbb{H}$, then
\[
\Delta_z \Phi(z, f) = c^+(0, 0).
\]
\end{theorem}

The function $\Phi(z, f)$ is a Green function for the divisor $Z(f) + C$, where $C$ is a divisor supported at the cusps.

\subsection{Non-holomorphic Hejhal--Poincar\'e Series}

Let $k \in \tfrac{1}{2}\mathbb{Z}$, and let $M_{\mu,\nu}(v)$ denote the standard Whittaker functions. Define
\[
\mathcal{M}_{s, k}(v) = v^{-\frac{k}{2}} M_{-\frac{k}{2}, s - \frac{1}{2}}(v).
\]

For $m > 0$ and $\mu \in L^\sharp/L$ with $m \equiv Q(\mu) \pmod{\mathbb{Z}}$, the non-holomorphic Hejhal--Poincar\'e series of index $(m, \mu)$ and weight $k$ is defined by \cite[Chapter 1]{Br} as
\[
F_{m,\mu}(\tau, s, k) = \frac{1}{2\Gamma(2s)} \sum_{\gamma \in \widetilde{\Gamma_\infty} \setminus \Gamma'} \left[ \mathcal{M}_{s, k}(4\pi m v) e(-mu) e_\mu \right] \mid_{k, \overline{\rho}_L} \gamma.
\]

Let $R_{k-2j}^j = R_{k-2} \circ \cdots \circ R_{k-2j}$ denote the $j$-fold composition of raising operators. Then
\[
\frac{1}{(4\pi m)^j} R_{k-2j}^j F_{m, \mu}(\tau, s_0 + j, k - 2j) = j! \, F_{m, \mu}(\tau, s_0 + j, k),
\]
where $s_0 = 1 - \frac{k}{2}$.

For $f \in H_{3/2 - \kappa, \overline{\rho}}$, define the twisted higher regularized theta lift by
\begin{equation}\label{equphij}
\Phi^j(z, f) = \frac{1}{(4\pi)^j} \times
\begin{cases}
\Phi(z, R^j_{1/2 - 2j} f), & \kappa = 2j + 1; \\
\Phi^{\mathcal{M}}(z, R^j_{-1/2 - 2j} f), & \kappa = 2j + 2.
\end{cases}
\end{equation}

This is a higher Green function for the divisor
\[
Z^j(f) = \sum_{m > 0,\, \mu \in L^\sharp/L} c^+(-m, \mu) m^j Z(m, \mu).
\]

Let $F_{m,\mu}(\tau) = F_{m,\mu}(\tau, s_0, k)$.

\begin{proposition}[\cite{BEY, Du}]\label{protwo}
Assume $F_{m,\mu} \in H_{3/2 - \kappa, \overline{\rho}}$. Then
\[
\Phi^j(z, F_{m,\mu}) = (-1)^{\kappa} \cdot 2m^{\frac{\kappa - 1}{2}} G_{N, \kappa}(z, Z(m, \mu)),
\]
where $F_{m,\mu} = F_{m,\mu}(\cdot, 1/4 + \kappa/2, 3/2 - \kappa)$.
\end{proposition}

\section{ Heegner cycles and height pairing}\label{Sec3}

\subsection{Heegner Cycles}

Let $\kappa$ be a positive integer. We recall the construction of Kuga-Sato varieties $\mathcal{Y} = \mathcal{Y}_\kappa(N)$ and their CM cycles, following \cite{Zhang}.

Let $D$ be a discriminant and $E$ an elliptic curve with complex multiplication by $\sqrt{D}$. Define the divisor
\[
Z(E) = \Gamma - (E \times \{0\}) + D(\{0\} \times E)
\]
on $E \times E$, where $\Gamma$ denotes the graph of multiplication by $\sqrt{D}$. Then $Z(E)^{\kappa-1}$ defines a cycle in $E^{2\kappa-2}$ of codimension $\kappa-1$.

Let $G_{2\kappa-2}$ be the symmetric group on $2\kappa-2$ letters, acting on $E^{2\kappa-2}$ by permuting factors. Define
\[
S_\kappa(E) = c \sum_{\sigma \in G_{2\kappa-2}} \mathrm{sgn}(\sigma) \, \sigma^*(Z(E)^{\kappa-1}),
\]
where $c \in \mathbb{R}$ is chosen so that the self-intersection of $S_\kappa(E)$ on each fiber equals $(-1)^{\kappa-1}$.

Now let $N$ be a product of two coprime integers $\geq 3$. The Kuga–Sato variety $\mathcal{Y}$ is defined as a canonical resolution of the $(2\kappa-2)$-fold fiber product of the universal elliptic curve $\mathcal{E}(N)$ over the modular curve $\mathcal{X}(N)$.

For a CM point $y \in \mathcal{X}(N)$, define the CM cycle $S_\kappa(y)$ over $y$ to be $S_\kappa(\mathcal{E}_y)$ in $\mathcal{Y}$.

Let $\pi: \mathcal{X}(N) \to \mathcal{X}_0(N)$ be the projection map. For a CM point $x \in \mathcal{X}_0(N)$, write
\[
\pi^*(x) = \frac{w(x)}{2} \sum x_i, \quad w(x) = |\mathrm{Aut}(x)|.
\]
Define the cycle over $x$ by
\begin{equation}\label{defcycle}
S_\kappa(x) = \frac{1}{\deg(\pi)} \sum S_\kappa(x_i).
\end{equation}

We now define the higher Heegner cycles:
\begin{align}
Z_{\kappa}(m, \mu) &= m^{\frac{\kappa-1}{2}} \sum_{x \in Z(m, \mu)} S_\kappa(x), \\
Z_{\kappa}(U) &= m_0^{\frac{\kappa-1}{2}} \sum_{x \in Z(U)} S_\kappa(x), \\
Z_{\kappa}(f) &= \sum_{m > 0} c^+(-m, \mu) Z_{\kappa}(m, \mu).
\end{align}
Here $U$ is identified with an imaginary quadratic field of fundamental discriminant $D_0 = -4Nm_0$.

\subsection{Moduli Stacks}

Let $\mathcal{X}_0(N)$ be the moduli stack over $\mathbb{Z}$ classifying cyclic isogenies $\pi : E \rightarrow E'$ of degree $N$ between generalized elliptic curves, such that $\ker \pi$ meets every irreducible component of each geometric fiber. According to \cite{KM}, the stack $\mathcal{X}_0(N)$ is regular, flat over $\mathbb{Z}$, and smooth over $\mathbb{Z}[\frac{1}{N}]$.

Let $D = -4Nm$ be a discriminant and $\mathcal{O}_D = \mathbb{Z}[\frac{D + \sqrt{D}}{2}]$ the order of discriminant $D$. Suppose $D \equiv r^2 \pmod{4N}$ and let $\mu = \begin{pmatrix} \frac{r}{2N} & 0 \\ 0 & -\frac{r}{2N} \end{pmatrix}$. Then $\mathfrak{n} = [N, \frac{r + \sqrt{D}}{2}]$ is an ideal of $\mathcal{O}_D$ with norm $N$.

Following \cite[Section 7]{BY}, define the moduli stack $\mathcal{Z}(m, \mu)$ over $\mathbb{Z}$ classifying pairs $(x, \iota)$, where:
\begin{enumerate}
    \item $x = (\pi : E \rightarrow E') \in \mathcal{Y}_0(N)$,
    \item $\iota : \mathcal{O}_D \hookrightarrow \operatorname{End}(x) = \{ \alpha \in \operatorname{End}(E) : \pi \alpha \pi^{-1} \in \operatorname{End}(E') \}$ is a CM action satisfying $\iota(\mathfrak{n}) \ker \pi = 0$.
\end{enumerate}
This stack is a Deligne–Mumford stack over $\mathbb{Z}$. The forgetful map
\[
\mathcal{Z}(m, \mu) \rightarrow \mathcal{X}_0(N), \quad (\pi : E \rightarrow E', \iota) \mapsto (\pi : E \rightarrow E')
\]
is finite and étale.

\begin{lemma}[\cite{BEY}, Lemma 6.10]
Let $c$ be the conductor of $\mathcal{O}_D$ and $N' = (N, c)$. Let $\bar{Z}(m, \mu)$ be the Zariski closure of $Z(m, \mu)$ in $\mathcal{X}_0(N)$. Then there is an isomorphism
\[
\mathcal{Z}(m, \mu) \cong \bar{Z}(m, \mu)
\]
of stacks over $\mathbb{Z}[\frac{1}{N'}]$.
\end{lemma}

Now let $D_0$ be a fundamental discriminant and $k = \mathbb{Q}(\sqrt{D_0})$ an imaginary quadratic field. Define the moduli stack $\mathcal{C}$ over $\mathbb{Z}$ such that for any scheme $S$, the set $\mathcal{C}(S)$ consists of pairs $(\mathcal{E}, \iota)$, where:
\begin{itemize}
    \item $\mathcal{E}$ is a CM elliptic curve over $S$,
    \item $\iota : \mathcal{O}_k \hookrightarrow \mathcal{O}_{\mathcal{E}} := \operatorname{End}_S(\mathcal{E})$ is an $\mathcal{O}_k$-action such that the main involution on $\mathcal{O}_{\mathcal{E}}$ induces complex conjugation on $k$.
\end{itemize}

There is a natural isomorphism of stacks \cite[Lemma 7.10]{BY}:
\[
j : \mathcal{C} \rightarrow \mathcal{Z}(m_0, \mu_0), \quad j((\mathcal{E}, \iota)) = (\pi : \mathcal{E} \rightarrow \mathcal{E}_{\mathfrak{n}_0}, \iota).
\]

\subsection{Height Pairing}

Let $x$ be a CM point in $X_0(N)$ and $\bar{x}$ its Zariski closure in $\mathcal{X}_0(N)$. Since the class of a Heegner cycle in $H^{2\kappa}(\mathcal{Y}(\mathbb{C}), \mathbb{C})$ is zero, there exists a Green current $g_\kappa(x)$ on $\mathcal{Y}(\mathbb{C})$ satisfying
\[
\frac{\partial \bar{\partial}}{\pi i} g_\kappa(x) = \delta_{S_\kappa(x)}.
\]

Following Gillet–Soulé \cite{GS}, define the codimension-$\kappa$ arithmetic cycle on $\mathcal{Y}$ by
\[
\hat{S}_\kappa(x) = (S_\kappa(\bar{x}), g_\kappa(x)).
\]

The intersection number is defined as
\[
\langle S_\kappa(x), S_\kappa(y) \rangle = (-1)^\kappa \langle \hat{S}_\kappa(x), \hat{S}_\kappa(y) \rangle_{\text{GS}}.
\]

By Zhang's work \cite{Zhang}, this global pairing decomposes into local heights:
\begin{equation}\label{definter}
\langle S_\kappa(x), S_\kappa(y) \rangle = \sum_{p \leq \infty} \langle S_\kappa(x), S_\kappa(y) \rangle_p.
\end{equation}

\section{Local intersections of Heegner cycles}\label{Sec4}

\subsection{Finite Intersection Numbers}

Let $\mathcal{Z}_\kappa(U) = \mathcal{Z}_\kappa(m_0, \mu_0)$ with fundamental discriminant $D_0 = -4Nm_0$, and recall that $\mathcal{Z}(U) \simeq \mathcal{C}$.

Consider the fiber product:
\[
\xymatrix{
j^*\mathcal{Z}(m_1, \mu_1) \ar[r]^{\pi_1} \ar[d]_{\pi_2} & \mathcal{C} \ar[d] \\
\mathcal{Z}(m_1, \mu_1) \ar[r] & \mathcal{X}_0(N)
}
\]
That is,
\[
j^*\mathcal{Z}(m_1, \mu_1) = \mathcal{Z}(m_1, \mu_1) \times_{\mathcal{X}_0(N)} \mathcal{C}.
\]

\begin{lemma}[\cite{BY}, Lemma 7.10]
There is an isomorphism:
\begin{equation}\label{equfirst}
j^*\mathcal{Z}(m_1, \mu_1)(\bar{\mathbb{F}}_p)
\cong \bigsqcup_{\substack{n \equiv r_0 r_1 \pmod{2N} \\ n^2 \leq D_0 D_1}}
\mathcal{Z}\left( \frac{D_0 D_1 - n^2}{4N|D_0|}, \mathbf{n}_0, \frac{n + r_1 \sqrt{D_0}}{2\sqrt{D_0}} \right)(\bar{\mathbb{F}}_p),
\end{equation}
where $\mathcal{Z}(m, \mathbf{a}, \mu)$ is an algebraic stack in $\mathcal{C}$; see \cite{KYPull, BY} for details.
\end{lemma}

Bruinier and Yang \cite[Section 6]{BY} showed that
\begin{equation}\label{equsec}
\widehat{\deg}(\mathcal{Z}(m, \mathbf{a}, \mu)) = -\frac{\deg(Z(U))}{2} \kappa(m, \mu).
\end{equation}

The $n$-th Legendre polynomial is defined by
\[
P_n(x) = \frac{1}{2^n n!} \frac{d^n}{dx^n}(x^2 - 1)^n.
\]

Let $x_i \in Z(m_i, \mu_i)$ and $\bar{x}_i$ its Zariski closure in $\mathcal{X}_0(N)$, for $i = 0, 1$. By \cite[Section 6]{BEY}:
\begin{eqnarray}\label{equthir}
\langle S_\kappa(x_1), S_\kappa(x_0) \rangle_p &=& (-1)^\kappa \langle S_\kappa(\bar{x}_1) \cdot S_\kappa(\bar{x}_0) \rangle_p \\
&=& (-1)^\kappa P_{\kappa-1}\left( \frac{n}{\sqrt{D_0 D_1}} \right) \langle \bar{x}_1 \cdot \bar{x}_0 \rangle_p, \nonumber
\end{eqnarray}
where over $p$, $\bar{x}_0$ and $\bar{x}_1$ intersect in $Z\left( \frac{D_0 D_1 - n^2}{4N|D_0|}, \mathbf{n}_0, \frac{n + r_1 \sqrt{D_0}}{2\sqrt{D_0}} \right)(\bar{\mathbb{F}}_p)$.

\begin{proposition}\label{profinite}
If $(D_0, 2N D_1) = 1$, then
\begin{eqnarray}
\langle \mathcal{Z}_\kappa(m_1, \mu_1), \mathcal{Z}_\kappa(U) \rangle_{\text{fin}} &=&
\frac{(-1)^{\kappa-1} (m_0 m_1)^{\frac{\kappa-1}{2}} \deg(Z(U))}{2}\nonumber \\
&\times& \sum_{\substack{n \equiv r_1 r_0 \pmod{2N} \\ n^2 \leq D_0 D_1}}
\kappa\left( \frac{D_0 D_1 - n^2}{4N|D_0|}, \frac{\tilde{2}n}{\sqrt{D_0}} \right)
P_{\kappa-1}\left( \frac{n}{\sqrt{D_0 D_1}} \right).
\end{eqnarray}
\end{proposition}
\begin{proof}
The case $\kappa$ odd is \cite[Proposition 6.9]{BEY}. For $\kappa$ even, combine equations (\ref{equfirst})--(\ref{equthir}).
\end{proof}

\subsection{Infinite Intersection and CM Values}

The generating function for $P_n(x)$ is
\[
G(x, t) = \sum_{n=0}^\infty P_n(x) t^n = \frac{1}{\sqrt{1 - 2xt + t^2}}.
\]

Define
\[
\beta_n(x) = \sum_{s=0}^{\lfloor n/2 \rfloor} \frac{n!}{2^{2s} (n - 2s)! (s!)^2} x^{n - 2s} (x^2 - 1)^s.
\]
For $n = 2j + 1$:
\[
\beta_{2j+1}(x) = \sum_{s=0}^j \binom{j + 1/2}{s} \binom{j}{s} x^{2j + 1 - 2s} (x^2 - 1)^s,
\]
and for $n = 2j$:
\[
\beta_{2j}(x) = \sum_{s=0}^j \binom{j - 1/2}{s} \binom{j}{s} x^{2j - 2s} (x^2 - 1)^s.
\]

Recall the hypergeometric function:
\[
{}_1F_0(\alpha; \ast; x) = \sum_{n=0}^\infty \frac{(\alpha)_n}{n!} x^n, \quad (\alpha)_n = \alpha(\alpha+1)\cdots(\alpha+n-1),
\]
which satisfies ${}_1F_0(\alpha; \ast; x) = (1 - x)^{-\alpha}$.

\begin{lemma}
For all $n$, $P_n(x) = \beta_n(x)$.
\end{lemma}
\begin{proof}
We compute:
\begin{eqnarray*}
G(x, t) &=& \left[ (1 - xt)^2 - t^2(x^2 - 1) \right]^{-1/2} \\
&=& (1 - xt)^{-1} \left[ 1 - \frac{t^2(x^2 - 1)}{(1 - xt)^2} \right]^{-1/2} \\
&=& (1 - xt)^{-1} \, {}_1F_0\left( \tfrac{1}{2}; \ast; \frac{t^2(x^2 - 1)}{(1 - xt)^2} \right) \\
&=& \sum_{s=0}^\infty \frac{(\frac{1}{2})_s}{s!} \frac{t^{2s}(x^2 - 1)^s}{(1 - xt)^{2s+1}} \\
&=& \sum_{n=0}^\infty \sum_{s=0}^\infty \frac{(\frac{1}{2})_s (2s+1)_n}{s! \, n!} x^n (x^2 - 1)^s t^{2s+n}.
\end{eqnarray*}
Extracting the coefficient of $t^n$:
\begin{eqnarray*}
P_n(x) &=& \sum_{s=0}^{\lfloor n/2 \rfloor} \frac{(\frac{1}{2})_s (2s+1)_{n-2s}}{s! \, (n-2s)!} x^{n-2s} (x^2 - 1)^s \\
&=& \sum_{s=0}^{\lfloor n/2 \rfloor} \frac{(\frac{1}{2})_s \, n!}{s! \, (n-2s)! \, (2s)!} x^{n-2s} (x^2 - 1)^s \\
&=& \sum_{s=0}^{\lfloor n/2 \rfloor} \frac{n!}{2^{2s} (s!)^2 (n-2s)!} x^{n-2s} (x^2 - 1)^s = \beta_n(x). \qedhere
\end{eqnarray*}
\end{proof}

\begin{remark}
The even case appears in \cite{BEY}; the above proof works for all $n$.
\end{remark}

Let $U \simeq k = \mathbb{Q}(\sqrt{D_0})$ with $D_0 = -4N m_0$ fundamental, $D_0 \equiv r_0^2 \pmod{4N}$. Decompose $V = U \oplus V_+$ and define lattices:
\[
\mathcal{N} = L \cap U, \quad \mathcal{P} = L \cap V_+.
\]
Then $\mathcal{N} \oplus \mathcal{P} \subset L$, and by \cite[Section 4]{BY}:
\[
\langle f, \Theta_L \rangle = \langle f_{\mathcal{P} \oplus \mathcal{N}}, \Theta_{\mathcal{P}} \otimes \Theta_{\mathcal{N}} \rangle.
\]
Assuming $L = \mathcal{N} \oplus \mathcal{P}$, we have the splitting as follows
\[
\Theta_L(\tau, z_U^\pm) = \Theta_{\mathcal{P}}(\tau) \otimes \Theta_{\mathcal{N}}(\tau, z_U^\pm).
\]

Define the Millson theta function:
\[
\tilde{\Theta}_{\mathcal{P}}(\tau) = \sum_{\delta \in \mathcal{P}^\sharp/\mathcal{P}} \sum_{\lambda \in \mathcal{P}_\delta}
p_{z_U^+}(\lambda) e(Q(\lambda)) e_\delta \in M_{3/2, \rho_{\mathcal{P}}}.
\]

For $f \in H_{3/2-\kappa, \bar{\rho}_L}$, let
\[
CT_f = \begin{cases}
CT\left( \langle f^+, [\tilde{\Theta}_{\mathcal{P}}, \mathcal{E}_{\mathcal{N}}]_j \rangle \right), & \kappa = 2j + 2, \\
CT\left( \langle f^+, [\Theta_{\mathcal{P}}, \mathcal{E}_{\mathcal{N}}]_j \rangle \right), & \kappa = 2j + 1.
\end{cases}
\]

\begin{definition}
For $f \in H_{3/2-\kappa, \overline{\rho}}$, define the CM value:
\begin{equation}\label{defcm}
\Phi^j(Z(U), f) = \frac{2}{w_k} \sum_{z \in \operatorname{supp}(Z(U))} \Phi^j(z, f),
\end{equation}
where $w_k = |\mathcal{O}_k^\times|$.
\end{definition}

For $g = \sum_{\mu} \sum_{m>0} b(m, \mu) q^m e_\mu \in S_{\frac{1}{2}+\kappa, \rho_L}$, the Shimura lift \cite{GKZ} is:
\[
Sh_{m_0, \mu_0}(g) = \sum_{n=1}^\infty \sum_{d \mid n} d^{\kappa-1} \left( \frac{D_0}{d} \right)
b\left( \frac{m_0 n^2}{d^2}, \frac{n}{d} \mu_0 \right) q^n \in S_{2\kappa}(\Gamma_0(N)).
\]

\begin{theorem}\cite{BEY, Du}\label{TheCM}
\begin{eqnarray*}
\Phi^j(Z(U), f) &=& \deg(Z(U)) \, CT_f \\
&& + (-1)^\kappa \frac{2^{4 - 2\kappa} \Gamma(\kappa - \frac{1}{2}) \sqrt{N}}{m_0^{\frac{\kappa-1}{2}} \pi^{\frac{1}{2} + \kappa}}
L'(Sh_{m_0, \mu_0}(g), \kappa).
\end{eqnarray*}
\end{theorem}
\begin{remark}
Case $\kappa = 2j+1$ is in \cite{BEY}; case $\kappa = 2j+2$ is in \cite{Du}.
\end{remark}

Note that $\deg(Z(U)) = \frac{4}{\operatorname{vol}(\widehat{\mathcal{O}}_k^\times)} = \frac{4 h_k}{w_k}$.

Following \cite[Section 7]{BY}, define sublattices:
\begin{equation}\label{equn}
\mathcal{N} = \mathbb{Z} e_1 \oplus \mathbb{Z} e_2, \quad
\mathcal{P} = \mathbb{Z} \cdot \frac{2N}{t} w,
\end{equation}
where $e_1 = \begin{pmatrix} 1 & 0 \\ -r_0 & -1 \end{pmatrix}$, $e_2 = \begin{pmatrix} 0 & \frac{1}{N} \\ \frac{r_0^2 - D_0}{4N} & 0 \end{pmatrix}$, and $t = (r_0, 2N)$.

\begin{proposition}
For $f_{m_1, \mu_1} \in H_{3/2-\kappa, \bar{\rho}_L}$,
\begin{eqnarray*}
CT_{f_{m_1, \mu_1}} &=& (-1)^{\kappa-1} 2 m_1^{\frac{\kappa-1}{2}} \\
&\times& \sum_{\substack{n \equiv r_1 r_0 \pmod{2N} \\ n^2 \leq D_0 D_1}}
\kappa\left( \frac{D_0 D_1 - n^2}{4N|D_0|}, \frac{\tilde{2}n}{\sqrt{D_0}} \right)
P_{\kappa-1}\left( \frac{n}{\sqrt{D_0 D_1}} \right).
\end{eqnarray*}
\end{proposition}
\begin{proof}
The case $\kappa = 2j+1$ is in \cite[Section 6]{BEY}. Assume $\kappa = 2j+2$. Then
\[
CT_{f_{m_1, \mu_1}} = CT\left( \langle f_{m_1, \mu_1}^+, [\tilde{\Theta}_{\mathcal{P}}, \mathcal{E}_{\mathcal{N}}]_j \rangle \right).
\]
By definition of the Rankin–Cohen bracket:
\begin{eqnarray*}
[\tilde{\Theta}_{\mathcal{P}}, \mathcal{E}_{\mathcal{N}}]_j &=& \sum a(n_1, v_1) \kappa(n_2, v_2) \\
&\times& \sum_{s=0}^j (-1)^s \binom{j - 1/2}{s} \binom{j}{s} n_1^{j-s} n_2^s q^{n_1 + n_2} e_{v_1} e_{v_2}.
\end{eqnarray*}
As in \cite[Proposition 6.7]{BEY}, the $(m_1, \mu_1)$-th coefficient is:
\begin{eqnarray*}
&& \sum_{\substack{n_1 + n_2 = m_1 \\ v_1 + v_2 = \mu_1}} a(n_1, v_1) \kappa(n_2, v_2)
\sum_{s=0}^j (-1)^s \binom{j + 1/2}{s} \binom{j}{s} n_1^{j-s} n_2^s \\
&=& -\sum_{\substack{n \equiv r_1 r_0 \pmod{2N} \\ n^2 \leq D_0 D_1}} \frac{n \sqrt{m_0}}{D_0}
\kappa\left( \frac{D_0 D_1 - n^2}{4N|D_0|}, \frac{\tilde{2}n}{\sqrt{D_0}} \right) \left( \frac{1}{4N|D_0|} \right)^j \\
&& \times \sum_{s=0}^j \binom{j + 1/2}{s} \binom{j}{s} n^{2j - 2s} (n^2 - D_0 D_1)^s.
\end{eqnarray*}
Using $p_{z_U^+}(w) = -\sqrt{m_0}$, this equals:
\begin{eqnarray*}
&& -\sqrt{D_0 D_1}^{\, 2j+1} \sum_{\substack{n \equiv r_1 r_0 \pmod{2N} \\ n^2 \leq D_0 D_1}}
\frac{\sqrt{m_0}}{D_0} \kappa\left( \frac{D_0 D_1 - n^2}{4N|D_0|}, \frac{\tilde{2}n}{\sqrt{D_0}} \right)
\left( \frac{1}{4N|D_0|} \right)^j \\
&& \times \sum_{s=0}^j \binom{j + 1/2}{s} \binom{j}{s}
\left( \frac{n}{\sqrt{D_0 D_1}} \right)^{2j+1 - 2s}
\left( \left( \frac{n}{\sqrt{D_0 D_1}} \right)^2 - 1 \right)^s \\
&=& -m_1^{\frac{2j+1}{2}} \sum_{\substack{n \equiv r_1 r_0 \pmod{2N} \\ n^2 \leq D_0 D_1}}
\kappa\left( \frac{D_0 D_1 - n^2}{4N|D_0|}, \frac{\tilde{2}n}{\sqrt{D_0}} \right)
\beta_{2j+1}\left( \frac{n}{\sqrt{D_0 D_1}} \right).
\end{eqnarray*}
Since $\beta_{2j+1}(-x) = -\beta_{2j+1}(x)$, the $(m_1, -\mu_1)$-coefficient gives the opposite value. The result follows.
\end{proof}

\section{Main results}\label{Sec5}
\subsection{Height Pairings}

For any CM point $x$ in $X_0(N)$, the Heegner cycle $S_\kappa(x)$ was defined by Zhang \cite{Zhang}. He constructed an arithmetic cycle on $\mathcal{Y}$ as 
\[
\hat{S}_\kappa(x) = (S_\kappa(\bar{x}), g_\kappa(x)).
\]

The height pairing is defined by
\[
\langle S_\kappa(x), S_\kappa(y) \rangle = (-1)^\kappa \langle \hat{S}_\kappa(x) \cdot \hat{S}_\kappa(y) \rangle_{\text{GS}},
\]
and decomposes into local contributions:
\begin{equation}\label{definter}
\langle S_\kappa(x), S_\kappa(y) \rangle = \sum_{p \leq \infty} \langle S_\kappa(x), S_\kappa(y) \rangle_p.
\end{equation}
Moreover,
\begin{eqnarray}\label{equarichi}
\langle S_\kappa(x), S_\kappa(y) \rangle_\infty &=& \frac{1}{2} G_{N, \kappa}(x, y), \\
\langle S_\kappa(x), S_\kappa(y) \rangle_p &=& (-1)^\kappa (S_\kappa(\bar{x}) \cdot S_\kappa(\bar{y}))_p,
\end{eqnarray}
where $G_{N, \kappa}(x, y)$ is the higher Green function.

For the cycles $Z_\kappa(f)$ and $Z_\kappa(m_0, \mu_0)$, we have:
\begin{eqnarray}\label{equneronfirst}
\langle Z_\kappa(f), Z_\kappa(m_0, \mu_0) \rangle 
&=& (-1)^\kappa \langle \widehat{\mathcal{Z}}_\kappa(f), \widehat{\mathcal{Z}}_\kappa(m_0, \mu_0) \rangle_{\text{GS}} \\
&=& \langle Z_\kappa(f), Z_\kappa(U) \rangle_\infty + \langle Z_\kappa(f), Z_\kappa(U) \rangle_{\text{fin}}. \nonumber
\end{eqnarray}

We now compute both terms.

\begin{lemma}\label{lemgreen}
The Green currents for the Heegner cycles $Z_\kappa(m, \mu)$ and $Z_\kappa(f)$ are given by
\[
(-1)^\kappa \cdot \frac{1}{2} \Phi^j(z, F_{m, \mu}) \quad \text{and} \quad (-1)^\kappa \Phi^j(z, f),
\]
respectively.
\end{lemma}
\begin{proof}
By equation (\ref{equarichi}), the Green current of a Heegner cycle is given by the higher Green function. 

For $Z_\kappa(m, \mu)$, the multiplicity $m^{\frac{\kappa - 1}{2}}$ implies that the current equals
\[
m^{\frac{\kappa - 1}{2}} G_{N, \kappa}(x, Z(m, \mu)).
\]
From Proposition \ref{protwo}, we have
\[
m^{\frac{\kappa - 1}{2}} G_{N, \kappa}(x, Z(m, \mu)) = (-1)^\kappa \cdot \frac{1}{2} \Phi^j(z, F_{m, \mu}).
\]
Since $Z_\kappa(f) = \sum_{m > 0} c^+(-m, \mu) Z_\kappa(m, \mu)$, linearity gives the Green current of $Z_\kappa(f)$ as $(-1)^\kappa \Phi^j(z, f)$.
\end{proof}

\begin{lemma}\label{leminfinite}
The archimedean contribution is given by
\begin{eqnarray*}
\langle Z_\kappa(f), Z_\kappa(U) \rangle_\infty 
&=& (-1)^\kappa \frac{m_0^{\frac{\kappa - 1}{2}} \deg(Z(U))}{2} CT_f \\
&& + \frac{2 \Gamma(\kappa - \frac{1}{2}) \sqrt{N} \, L'(Sh_{m_0, \mu_0}(\xi_{3/2 - \kappa} f), \kappa)}{(4\pi)^{\kappa - 1} \pi^{\frac{3}{2}}}.
\end{eqnarray*}
\end{lemma}
\begin{proof}
By Lemma \ref{lemgreen}, we have
\[
\langle Z_\kappa(f), Z_\kappa(U) \rangle_\infty = (-1)^\kappa \cdot \frac{m_0^{\frac{\kappa - 1}{2}}}{2} \Phi^j(Z(U), f).
\]
According to Theorem \ref{TheCM}, we have
\begin{eqnarray*}
\Phi^j(Z(U), f) &=& \deg(Z(U)) CT_f \\
&& + (-1)^\kappa \frac{2^{4 - 2\kappa} \Gamma(\kappa - \frac{1}{2}) \sqrt{N}}{m_0^{\frac{\kappa - 1}{2}} \pi^{\frac{1}{2} + \kappa}} L'(Sh_{m_0, \mu_0}(g), \kappa).
\end{eqnarray*}
Substituting and simplifying yields the result.
\end{proof}
\subsection{Global height}
\begin{theorem}\label{mainthe}
For any $f \in H_{3/2 - \kappa, \bar{\rho}_L}$ and Heegner cycle $Z_\kappa(U)$, the global height pairing is:
\begin{equation}
\langle Z_\kappa(f), Z_\kappa(U) \rangle = \frac{2 \sqrt{N} \, \Gamma(\kappa - \frac{1}{2})}{(4\pi)^{\kappa - 1} \pi^{\frac{3}{2}}} L'(Sh_{m_0, \mu_0}(\xi_{3/2 - \kappa} f), \kappa).
\end{equation}
\end{theorem}
\begin{proof}
By equation (\ref{equneronfirst}), we have
\[
\langle Z_\kappa(f), Z_\kappa(U) \rangle = \langle Z_\kappa(f), Z_\kappa(U) \rangle_\infty + \langle Z_\kappa(f), Z_\kappa(U) \rangle_{\text{fin}}.
\]
Proposition \ref{profinite} gives 
\[
\langle Z_\kappa(f), Z_\kappa(U) \rangle_{\text{fin}} = (-1)^{\kappa - 1} \frac{m_0^{\frac{\kappa - 1}{2}} \deg(Z(U))}{2} CT_f.
\]
Combining with Lemma \ref{leminfinite}, the terms involving $CT_f$ cancel, and the result follows.
\end{proof}

As an application, we obtain:

\begin{corollary}
For any $f \in M^!_{3/2 - \kappa, \bar{\rho}_L}$, we have 
\[
\langle Z_\kappa(f), Z_\kappa(U) \rangle = 0.
\]
\end{corollary}

\section{Modularity}\label{Sec6}

In this section, we discuss the modularity conjectures proposed in this work.

\subsection{Modularity of Generating Series}

Recall that for $f \in H_{3/2-\kappa,\overline{\rho}}$, the higher regularized theta lift is defined by
\begin{equation}\label{equphij}
\Phi^j(z, f) = \frac{1}{(4\pi)^j} \times
\begin{cases}
\Phi(z, R^j_{1/2 - 2j} f), & \kappa = 2j + 1, \\
\Phi^{\mathcal{M}}(z, R^j_{-1/2 - 2j} f), & \kappa = 2j + 2.
\end{cases}
\end{equation}

Define the normalized lift  as follows
\begin{equation}\label{equnormlift}
\Phi^{(\kappa)}(z, f) = \frac{i^\kappa N^{\frac{\kappa - 1}{2}}}{4(\kappa - 1)!} R_{0, z}^\kappa \Phi^j(z, f).
\end{equation}

\begin{proposition}[\cite{ABS}, Section 4]
 The function $\Phi^{(\kappa)}(z, f) $ is a meromorphic modular form of weight $2\kappa$, and it has the Fourier expansion 
\[
\Phi^{(\kappa)}(z, f) = C_\kappa \pi^\kappa i \sum_{n \geq 1} \sum_{d \mid n} d^{-\kappa} c^+\left( \frac{n^2}{d^2}, \frac{n}{d} \right) q^n,
\]
where $C_\kappa = \frac{(-2)^\kappa}{(\kappa - 1)!}$. Moreover, the residue is given by
\[
\operatorname{res}(\Phi^{(\kappa)}(z, f)) = |4N|^{\frac{\kappa - 1}{2}} \sum_{m > 0,\, \mu \in L^\sharp/L} c^+(-m, \mu) m^{\frac{\kappa-1}{2}} Z(m, \mu).
\]
\end{proposition}
If all $c(n, \mu) \in K$, then all coefficients of $\Phi^{(\kappa)}(z, f)$ belong to $\pi^\kappa i K$.

Let $F$ be a meromorphic modular form of weight $2\kappa$. Suppose that at each pole $\varrho \in \mathbb{H}$, it admits the local expansion
\[
F(z) = a_\varrho \left( \frac{(z - \varrho)(z - \bar{\varrho})}{\varrho - \bar{\varrho}} \right)^{-\kappa} + O(1), \quad z \to \varrho.
\]
Following \cite{ABS}, the residue divisor of $F$ is defined as
\[
\operatorname{res}(F) = \sum_{[\varrho] \in Y_0(N)} \frac{a_\varrho}{w_\varrho} [\varrho] \in \operatorname{Div}(X_0(N)),
\]
where $w_\varrho$ denotes half the order of the stabilizer of $\varrho$ in $\Gamma_0(N)$.

Let $\mathrm{Heeg}_\kappa(X)$ be the group generated by the Heegner cycles $Z_\kappa(m, \mu)$, and let $\mathrm{Pin}$ be the subgroup of principal cycles. A divisor is called \emph{algebraic principal} if there exists an algebraic modular form of weight $2\kappa$ whose residue divisor coincides with it. In analogy with \cite[Definition 4.4]{Ze}, we say that a cycle $Z \in \mathrm{Heeg}_\kappa(X)$ is \emph{algebraic principal} if it is defined over such an algebraic principal divisor.

Let $\mathrm{Pin}^{\mathrm{alg}}$ denote the subgroup of algebraic principal cycles. We define the quotient groups
\[
\mathrm{Heeg}_\kappa(\mathrm{Cl}(X)) = \mathrm{Heeg}_\kappa(X) / \mathrm{Pin}, \qquad
\mathrm{Heeg}_\kappa^{\mathrm{alg}}(\mathrm{Cl}(X)) = \mathrm{Heeg}_\kappa(X) / \mathrm{Pin}^{\mathrm{alg}}.
\]

Let $F$ be a meromorphic modular form of weight $2\kappa$. Assume that at each
pole $\varrho \in \H$,  it has the following expansions
$$F(z)=a_\rho \bigg(\frac{(z-\rho)(z-\bar\rho)}{\rho -\bar\rho}\bigg)^{-\kappa}+O(1), z\rightarrow \rho.$$
Then the residue divisor is defined in \cite{ABS} as 
$$res(F)=\sum_{[\rho] \in Y_0(N)}\frac{a_p}{w_\rho}[\rho] \in Div(X_0(N)).$$
Here $w_\rho$ is half the order of the stabilizer of $\rho$ in $\Gamma_0(N)$.

Let $\mathrm{Heeg}_{\kappa}(X)$ be the group generated by Heegner cycles $Z_\kappa(m, \mu)$, and let $\mathrm{Pin}$ be the subgroup of principal cycles.  
A divisor is called algebraic principal if there exists an algebraic  modular form of weight $2\kappa$ whose residue  divisor  coincides with it. Analogously to \cite[Definition 4.4]{Ze}, we say that a cycle $Z \in \mathrm{Heeg}_{\kappa}(X)$ is algebraic principal if it is defined over such an algebraic principal divisor.
Let $\mathrm{Pin}^{\mathrm{alg}}$ denote the subgroup of algebraic principal cycles, and define the quotient groups by
\[
\mathrm{Heeg}_{\kappa}(\mathrm{Cl}(X)) = \mathrm{Heeg}_{\kappa}(X) / \mathrm{Pin}, \quad
\mathrm{Heeg}_{\kappa}^{\mathrm{alg}}(\mathrm{Cl}(X)) = \mathrm{Heeg}_{\kappa}(X) / \mathrm{Pin}^{\mathrm{alg}}.
\]

Let $G \in S_{2\kappa}^{\mathrm{new}}(N)$ be a normalized newform, and let $F_G$ be the totally real number field generated by the eigenvalues of $G$. Under the Shimura correspondence $\mathrm{Sh}_{m_0, \mu_0}$, there exists a newform $g \in S_{\frac{1}{2} + \kappa, \rho_L}^{\mathrm{new}}$ corresponding to $G$, with all Fourier coefficients in $F_G$. Moreover, there exists $f \in H_{3/2 - \kappa, \bar{\rho}_L}(F_G)$ such that
\[
\xi_{3/2 - \kappa}(f) = \| g \|^{-2} g,
\]
and all coefficients of the principal part of $f$ lie in $F_G$.

It is conjectured in \cite{ABS} that $Z_\kappa(f)$ vanishes (i.e., $Z_\kappa(f) \in \mathrm{Pin}$) if and only if all coefficients $c^+(m^2, m)$ belong to $F_G$.

Now let $f$ be a weakly holomorphic form of weight $3/2 - \kappa$. Then
\[
R^j_{3/2 - \kappa} f(\tau) = \sum_{m=0}^j\sum_{n, \mu} c(n, \mu, m) \frac{q^n}{v^m} e_\mu
\]
is an almost holomorphic modular form, with coefficients:
\[
c(n, \mu, m) = (-1)^j (4\pi)^{j - m} \binom{j}{m} n^{j - m} \prod_{s = 1}^m (k - j - s) c(n, \mu).
\]
In particular, the holomorphic part satisfies:
\begin{equation}\label{equcoefhol}
c(n, \mu, 0) = (-4\pi)^j n^j c(n, \mu).
\end{equation}

\begin{proposition}\label{theweak}
Let $f$ be a weakly holomorphic form with algebraic coefficients in a number field. Then $Z_\kappa(f)$ is algebraic principal, i.e., $Z_\kappa(f) \in \mathrm{Pin}^{\mathrm{alg}}$.
\end{proposition}
\begin{proof}
We follow the proof in \cite{Ze, ABS}.
Assume that the coefficients of $f$ are contained in $K$. Then coefficients of $\Phi^{(\kappa)}(z, f)$ are contained in $\pi^\kappa i K$. Thus $res(\Phi^{(\kappa)}(z, f))$ is algebraic principal,
and hence $Z_\kappa(f) \in \mathrm{Pin}^{\mathrm{alg}}$.
\end{proof}

\begin{lemma}\label{lemcomp}
Let the notation be as above. 
The cycle $Z_\kappa(f)$ belongs to the $G$-component of $\mathrm{Heeg}_{\kappa}^{\mathrm{alg}}(\mathrm{Cl}(X))$.
\end{lemma}
\begin{proof}
By the same method as in \cite[Theorem 7.5]{BO}, we obtain the result.
\end{proof}

In his work on modularity problems, Borcherds introduced an innovative approach via Serre duality in \cite{BoDuke}. He defined a pairing $\{~,~\}$ between formal power series and formal Laurent series, and established the following key criterion: a formal power series $g$ is a cusp form if and only if $\{f, g\} = 0$ for every weakly holomorphic modular form $f$.

The generation function is defined by 
\[
\phi_\kappa = \sum_{m > 0, \mu} Z_\kappa(m, \mu) q^m e_\mu .
\]
Then we obtain:

\begin{theorem}\label{themainone}
The following hold:
\begin{enumerate}
    \item $\phi_\kappa \in \mathrm{Heeg}_{\kappa}^{\mathrm{alg}}(\mathrm{Cl}(X)) \otimes S_{1/2 + \kappa, \rho_L}$ is a cusp form.
    \item The $G$-component is given by
    \[
    \phi_\kappa^G = \sum_{m > 0, \mu} Z_\kappa^G(m, \mu) q^m e_\mu= g(\tau) \otimes Z_\kappa(f).
    \]
\end{enumerate}
\end{theorem}
\begin{proof}
(1)
By Proposition \ref{theweak}, for any weakly holomorphic modular form  $f$, the paring $\{ f, \phi_\kappa \} = Z_\kappa(f)$ vanishes in $\mathrm{Heeg}_{\kappa}^{\mathrm{alg}}(\mathrm{Cl}(X))$. According to 
the  Serre duality, we know that $\phi_\kappa$ is a cusp form.

(2)
 As in \cite[Theorem 7.7]{BO}, $\phi_\kappa^G$ and $g$ share the same eigenvalues. By Lemma \ref{lemcomp}, $Z_\kappa(f)$ lies in the $G$-component. Thus we have
\[
\phi_\kappa^G = g(\tau) \otimes Z_\kappa(f).
\]
Then we complete the proof.
\end{proof}

\begin{conjecture}
Theorem \ref{themainone} remains valid when $\mathrm{Heeg}_{\kappa}^{\mathrm{alg}}(\mathrm{Cl}(X))$ is replaced by $\mathrm{Heeg}_{\kappa}(\mathrm{Cl}(X))$.
\end{conjecture}

By the previous proof, it suffices to show that if $f$ is a weakly holomorphic form with algebraic coefficients, then $Z_\kappa(f)$ is principal.

Moreover, we have the following conjecture:
 \begin{conjecture} $\mathrm{Pin}^{\mathrm{alg}} \subseteq \mathrm{Pin}$.\end{conjecture}

\begin{conjecture} 
  The arithmetic generating function
\begin{equation}
\hat{\phi}\kappa = \sum_{m>0, \mu} \widehat{\mathcal{Z}}\kappa(m, \mu) q^m e_\mu \in \widehat{\mathrm{CH}}^{\kappa}(\mathcal{Y}) \otimes \mathbb{C}[[q]]
\end{equation}
is also a cusp form. Here $\widehat{\mathrm{CH}}^{\kappa}(\mathcal{Y})$ is the arithmetic Chow group.
\end{conjecture}
For $\kappa=1$, the author proved the modularity of the arithmetic generating function on the modular curve $X_0(N)$ in \cite{DY}.

\subsection{Related works}

We define the space
\begin{equation}
 W^G=span\{Z_{\kappa}^G(m, \mu)|(4Nm, 2N)=1\}.
\end{equation}

\begin{theorem}\cite[Theorem 2]{Xue}
Assume that the  height pairing on Heegner cycles is positive definite.
Then the space $W^G$ has dimension at most $1$.
\end{theorem}
\begin{remark}
We conjecture  that $W^G$ is generated by $Z_\kappa(f)$.
\end{remark}

\begin{proposition}
\[
\text{Theorem \ref{mainthe}} \quad \overset{\text{Modularity}}{\Longrightarrow} \quad \text{Gross--Zagier--Zhang formula}.
\]
\end{proposition}
\begin{proof}
We choose $Z_\kappa(m, \mu) = Z_\kappa(U)$, such that $c(m, \mu)\neq 0$.

Assuming modularity, we have $\phi_\kappa^G = g(\tau) \otimes Z_\kappa(f)$, and hence
\begin{eqnarray*}
\langle Z_\kappa(m, \mu)^G, Z_\kappa(m, \mu)^G \rangle &=& c(m, \mu) \langle Z_\kappa(f), Z_\kappa(m, \mu)^G \rangle \\
&=& c(m, \mu) \langle Z_\kappa(f), Z_\kappa(U) \rangle.
\end{eqnarray*}
By Theorem \ref{mainthe}, we obtain
\begin{eqnarray*}
\langle Z_\kappa(m, \mu)^G, Z_\kappa(m, \mu)^G \rangle = c(m, \mu) \frac{2\sqrt{N} \, \Gamma(\kappa - \frac{1}{2})}{(4\pi)^{\kappa - 1} \pi^{\frac{3}{2}}} L'(Sh_{m_0, \mu_0}(\xi_{3/2 - \kappa} f), \kappa).
\end{eqnarray*}
Combining this with the Waldspurger-type formula in \cite[Chapter 2]{GKZ} or \cite{Du},  we deduce the Gross-Zagier-Zhang formula.
\end{proof}

\section*{Acknowledgements}We would like to thank T. Yang and J. Bruinier for  useful suggestions.



\begin{thebibliography}{BCDE}


	
\bibitem[ABS] {ABS}{\em C. Alfes,   J. Bruinier, M. Schwagenscheidt}, Harmonic weak Maass forms and periods II,
Math. Ann., 391(2024),  791-817.

\bibitem[Bo1]{Bo}{\em R. Borcherds}, Automorphic forms with singularities on Grassmannians, Invent. Math., 132(1998), 491-562.

\bibitem[Bo2]{BoDuke} {\em R. Borcherds}, The Gross-Kohnen-Zagier theorem in higher dimensions, Duke Math. J. 97 (1999),  219-233.

\bibitem[Bost]{Bost}{\em J. Bost}, Potential theory and Lefschetz theorems for arithmetic surfaces, Ann. scient. \'Ec. Norm. Sup. 32(1999), 241-312.


\bibitem[Br]{Br}{\em J. Bruinier }, Borcherds products on $O(2, l)$ and Chern classes of Heegner divisors, Springer Lecture Notes in Mathematics, Springer-Verlag, 1780(2002).


 
 \bibitem[BEY]{BEY}{\em J. Bruinier, S. Ehlen, T. Yang}, CM values of higher automorphic Green functions
for orthogonal groups, Invent. math.,  (2021) 225, 693–785.




\bibitem[BY]{BY}{\em J. Bruinier, T. Yang}, Faltings heights of CM cycles and derivatives of L-functions, Invent. Math., 177 (2009), 631-681.






\bibitem[BO]{BO} {\em J. Bruinier, K. Ono}, Heegner divisors, L-functions and harmonic
weak maass forms, Ann.
Math., 172(2010), 2135–2181.



\bibitem[Du]{Du} {\em T. Du}, On the arithmetic product formula and central derivatives of L-functions, 	arXiv:2412.00688.

\bibitem[DY]{DY}  {\em T. Du, T. Yang}, Arithmetic Siegel-Weil formula on $X_0(N)$,
 Advances in mathematics, 345(2019), 702-755.





\bibitem[GKZ]{GKZ} {\em B. Gross, W. Kohnen, D. Zagier}, Heegner points and derivatives of L-series. II. Math. Ann.,
278 (1987), 497-562.

\bibitem[GS]{GS} {\em H. Gillet, C. Soul\'e},
Arithmetic intersection theory, Publ. Math. I.H.E.S.,
72(1990), 94-174.



\bibitem[GZ]{GZ} {\em  B. Gross, D. Zagier},  Heegner points and derivatives of L-series, Invent. Math., 84 (1986),  225--320.



\bibitem[KM]{KM}  {\em N. Katz, B. Mazur},
Arithmetic moduli of elliptic curves.Annals of Mathematics Studies, 108, Princeton University Press, Princeton, NJ, (1985).

\bibitem[KuM]{KM2} {\em S. Kudla, J. Millson}, Intersection numbers of cycles on
locally symmetric spaces and Fourier coefficients of holomorphic modular
forms in several complex variables, Inst. Hautes  \'Etudes Sci. Publ. Math., 71(1990), 121–172.



\bibitem[Ku]{KuIntegral} {\em S. Kudla}, Integrals of Borcherds forms,  Compositio Math., 137 (2003),  293-349.


\bibitem[KY]{KYPull}{\em S. Kudla, T. Yang}, On the pullback of an arithmetic theta function, manuscripta math., 140(2013), 393-440.

\bibitem[Sch]{Scho}  {\em J. Schofer}, Borcherds forms and generalizations of singular moduli, J. Reine Angew. Math., 629(2009), 1-36.


\bibitem[Xue]{Xue} {\em H. Xue}, Gross-Kohnen-Zagier theorem for higher weight   forms, Math. Res. Lett. 17 (2010), no. 3, 573-586. 
    
\bibitem[Ze]{Ze} {\em S. Zemel}, A Gross–Kohnen–Zagier type theorem for
higher-codimensional Heegner cycles, Research in Number Theory (2015) 1:23

 \bibitem[Zh]{Zhang} {\em S. Zhang}, Heights of Heegner cycles and derivatives of L-series, Invent. Math., 130(1997),
99-152.









%
%


    

\end{thebibliography}
\end{document}